\documentclass[12pt,reqno]{amsart}

\usepackage{graphicx}

\newtheorem{theorem}{Theorem}[section]   
\newtheorem{lemma}[theorem]{Lemma}         
\newtheorem{proposition}[theorem]{Proposition}  
\newenvironment{remark}{{\bf Remark.\ }\rm}{\bigskip}

\renewenvironment{abstract}{\begin{quote}{\bf Abstract.\ }\small}{\end{quote}\bigskip}

\newcommand{\bbinom}[2]{\genfrac{(}{)}{0pt}{0}{#1}{#2}}
\newcommand{\ff}[2]{\phi^{#1}_{#2}}
\newcommand{\FF}[2]{\Phi^{#1}_{#2}}

\def\I{{\rm I}}
\def\J{{\rm J}}
\def\B{\mathcal{B}}

\title[Sums of Squares for Krawtchouk]{Sums of squares of Krawtchouk polynomials, Catalan numbers, and some algebras over the Boolean lattice}
\author{Philip Feinsilver}
\address{Department of Mathematics\\ Southern Illinois University \\Carbondale, IL}
\email{phfeins@siu.edu}

\begin{document}
\maketitle
\thispagestyle{empty}
\date{}
\begin{abstract}
Writing the values of Krawtchouk polynomials as matrices, we consider weighted partial sums along columns. For the general case, we find
an identity that, in the symmetric case yields a formula for such partial sums. Complete sums of squares along columns involve
``Super Catalan" numbers. We look as well for particular values (matrix entries) involving the Catalan numbers. Properties
considered and developed in this work are applied to calculations of various dimensions that describe the structure of
some *-algebras over the Boolean lattice based on inclusion/superset relations expressed algebraically using zeons [zero-square elements].
\end{abstract}
{\footnotesize
{\tt Subject Classification[2010]: }{05A10, 06E25, 33C45, 42C05}\\
{\tt Keywords: }{Krawtchouk polynomials, Catalan numbers, Super Catalan numbers, Boolean lattice, *-algebras, zeons}
}
\begin{section}{Introduction}
Our approach to Krawtchouk polynomials is to consider the values as entries in corresponding matrices.
This makes it convenient to refer to their values, indexing, and associated properties. We begin with Krawtchouk	polynomials
for general parameter $p$ and derive an identity for partial sums of squares along a column. For the case $p=1/2$ this 
leads to evaluations of these sums. We also review some basic properties of Krawtchouk matrices that we will find useful.\\

Two related features are considered in detail. First is the ``Catalan connection" which appears when looking at complete sums of squares along columns.
As well, Catalan numbers appear as particular entries in Krawtchouk matrices. The second main feature is the use
of properties of Krawtchouk matrices in calculating dimensions of some algebras over the Boolean lattice. These arise when considering
the lattice of subsets of $\{1,2,\ldots,n\}$ as an algebra generated by ``zeons" --- commuting elements which square to zero.
These algebras are generated by the basic matrices corresponding to inclusion and superset, specifically, 
the regular representation of the zeon algebra.\\

References to Catalan numbers are readily available and abundant, so these have been skipped, but see \cite{GMT} for a relevant discussion
involving Super Catalan numbers. For sums of squares, we found
the results of \cite{D1} important, though we stress that the difference in
our approaches is substantial. Our approach to Krawtchouk	
polynomials follows \cite{FF,FK1}. For the Boolean connection, full details are presented in \cite{F1}. And we found the discussion in
\cite{Sag} especially valuable for background.

\end{section}
\begin{section}{Krawtchouk polynomials}
We modify the generating function given in \cite[18.23]{DLMF}, formula 18.23.3,
for convenience in computations as well as in point of view.
Throughout, we will use the parameter
$$ r=\frac{1-p}{p}$$
with $r=1$ corresponding to the symmetric case $p=1/2$. Note the relation
$$\frac{1}{p}=1+r$$
that we will find useful.\\

The generating function takes the form
\begin{equation}\label{eq:genf}
(1+z)^{N-j}(1-rz)^j=\sum_{n=0}^N z^n\,\phi_{nj}^{N}
\end{equation}
with slight changes in notation. In particular, we prefer the matrix form $\phi_{nj}^{N}$ for the values of the Krawtchouk polynomials at integer points, $0\le  j \le N$.

\begin{subsection}{Relations of Pascal type}
First we note some identities similar to Pascal's triangle for binomial coefficients. See \cite{FK1} for the case $r=1$.

\begin{proposition}\label{prop:Pascal}
The following identities of Pascal type hold for \hfill\break $N\ge 0$, $0\le j, n \le N$ :
\begin{align*}
\qquad \ff N{n  \;j} + \ff N{n-1\;j}   &=  \ff {N+1}{n\;j}    &&&\mathrm{(i)} \\
\mathstrut \nonumber\\
\qquad \ff N{n  \;j} -r\, \ff N{n-1  \;j} &= \ff {N+1}{n\;j+1}   &&&\mathrm{(ii)}
\end{align*}
with the boundary conditions $\ff N{-1\;j}=0$, $0\le  j \le N$, for $N\ge 0$. 
\end{proposition}
\begin{proof}
The first relation follows upon multiplication of the generating function \eqref{eq:genf} by $1+z$. The second follows similarly using
the factor $1-rz$. \hfill
\end{proof}
\end{subsection}

\begin{subsection}{Recurrence formula}
The second ingredient needed is a recurrence formula for Krawtchouk polynomials. From \cite[18.22]{DLMF}, formula 18.22.12, 
Difference Equations in $x$,
we have, replacing $x$ by $j$ and rearranging:
$$p(N-j)\ff N{n\;j+1}+qj \ff N{n\;j-1}=(Np+(q-p)j-n)\ff N{n\;j} $$
where we introduce the notation $q=1-p$ for convenience, with $r=q/p$. Dividing through by $p$ and noting $r=q/p$, $1/p=1+r$, let us state
\begin{lemma}\label{lem:recj}
We have the recurrence in $j$
$$
(N+(r-1)j-n(1+r))\ff N{n\;j}=(N-j)\ff N{n\;j+1}+rj \ff N{n\;j-1} \ .
$$
\end{lemma}
\end{subsection}
\end{section}

\begin{section}{Sums of Squares}
We now derive our main result.
\begin{subsection}{Sums of squares for general $r$}
Multiply equations (i) and (ii) of Proposition \ref{prop:Pascal}, for $N \rightarrow N-1$, to get
\begin{align*}
\ff {N}{n\;j+1}\ff {N}{n\;j} &=(\ff {N-1}{n\;j})^2-r (\ff {N-1}{n-1\;j})^2+(1-r)\ff {N-1}{n\;j}\ff {N-1}{n-1\;j}\\
&=(\ff {N-1}{n\;j})^2-r (\ff {N-1}{n-1\;j})^2+\frac{1-r}{2}\left[(\ff {N}{n\;j})^2-(\ff {N-1}{n\;j})^2-(\ff {N-1}{n-1\;j})^2\right]\\
&=\frac{1-r}{2}(\ff {N}{n\;j})^2+\frac{1+r}{2}\left[(\ff {N-1}{n\;j})^2-(\ff {N-1}{n-1\;j})^2\right]
\end{align*}
where in the second line we use Prop \ref{prop:Pascal}, (i), in the elementary identity $ab/2=(a+b)^2-a^2-b^2$. 
A similar formula holds replacing $j\rightarrow j-1$. Now multiply through the relation in Lemma \ref{lem:recj} by $\ff {N}{n\;j}$ to get
\begin{align*}
(N&+(r-1)j-n(1+r))(\ff N{n\;j})^2=(N-j)\ff N{n\;j+1}\ff N{n\;j}+rj \ff N{n\;j-1} \ff N{n\;j}\\
&=(N-j)\left(\frac{1-r}{2}(\ff {N}{n\;j})^2+\frac{1+r}{2}\left[(\ff {N-1}{n\;j})^2-(\ff {N-1}{n-1\;j})^2\right]\right)\\
&+rj\left(\frac{1-r}{2}(\ff {N}{n\;j-1})^2+\frac{1+r}{2}\left[(\ff {N-1}{n\;j-1})^2-(\ff {N-1}{n-1\;j-1})^2\right]\right)
\end{align*}
With the telescoping parts collapsing, we sum $n$ from $0$ to $m$ to get
\begin{align*}
\sum_{n=0}^m (N&+(r-1)j-n(1+r))(\ff N{n\;j})^2=\frac{1-r}{2}(N-j)\sum_{n=0}^m (\ff {N}{n\;j})^2+\frac{1+r}{2}(N-j) (\ff {N-1}{m\;j})^2      \\
&+rj\frac{1-r}{2}\sum_{n=0}^m (\ff {N}{n\;j-1})^2+rj\frac{1+r}{2}(\ff {N-1}{m\;j-1})^2
\end{align*}

Multiplying through by $2/(1+r)$ we arrive at our main formula.

\begin{theorem} \label{thm:sqsum} We have the sum of squares identity for Krawtchouk polynomials
\begin{align*}
\sum_{n=0}^m (N-2n)(\ff N{n\;j})^2&=(N-j)(\ff {N-1}{m\;j})^2+rj(\ff {N-1}{m\;j-1})^2\\
&\qquad +\frac{1-r}{1+r}\,j\,\sum_{n=0}^m \left(r(\ff N{n\;j-1})^2+(\ff N{n\;j})^2\right)
\end{align*}

\end{theorem}

Dette \cite{D1} has similar formulas, his formula (d) for Krawtchouk polynomials is most similar to ours.
\end{subsection}
\end{section}
\begin{section}{Symmetric case}
Letting $r=1$ gives the symmetric case with generating function
\begin{equation}\label{eq:symmgf}
(1+z)^{N-j}(1-z)^j=\sum_{n=0}^N z^n\,\Phi_{nj}^{N}
\end{equation}
with the capital $\Phi$ denoting the values for this special case.\\

The recurrence is now
\begin{equation}\label{eq:rrec}
(N-2n)\FF N{n\;j}=(N-j)\FF N{n\;j+1}+j \FF N{n\;j-1} 
\end{equation}

\begin{subsection}{Basic properties}
Here we recall some basic properties of the Kravchuk matrices for $r=1$.\\

\begin{proposition}\label{prop:props}\hfill \\
\rm
\begin{enumerate}     
\item Row and column sign symmetries.
\begin{align*}
\FF{N}{i\;N-j}&=(-1)^i\,\FF{N}{i\;j}\\
\FF{N}{N-i\;j}&=(-1)^j\,\FF{N}{i\;j}\\
\FF{N}{N-i\;N-i}&=(-1)^N\,\FF{N}{i\;i}\\
\end{align*}

The first two follow readily from the generating function, the third follows from those.\\

For reference, here are the matrices for $N=3$ and $N=4$:
$$
\Phi^{3}=
       \left[\begin{array}{rrrr}
                      1 &  1 &  1  &  1 \cr
                      3 &  1 & -1  & -3 \cr
                      3 & -1 & -1  &  3 \cr
                      1 & -1 &  1  & -1 \cr 
             \end{array}\right]
\qquad
\Phi^{4}=
       \left[\begin{array}{rrrrr} 1 &  1 &  1  &  1  &  1 \cr
                      4 &  2 &  0  & -2  & -4 \cr
                      6 &  0 & -2  &  0  &  6 \cr
                      4 & -2 &  0  &  2  & -4 \cr
                      1 & -1 &  1  & -1  &  1 \cr \end{array}\right]
$$

\item First rows. First columns.
The entries in row $i=1$ follow as the coefficient of $z$ in the expansion of the generating function
eq. \eqref{eq:symmgf} yielding
$$\FF{N}{1\;j}=N-2j\ .$$
The entries in the first column are the binomial coefficients 
$$\FF{N}{n\;0}=\binom{N}{n}\ .$$
 For the second column, we set $j=0$ in Proposition \ref{prop:Pascal}, (ii) to get
\begin{align}\label{eq:columnone}
\Phi_{n\;1}^{N+1}&=\Phi_{n\;0}^N-\Phi_{n-1\;0}^N=\binom{N}{n}-\binom{N}{n-1}\\ 
&=\binom{N}{n}\,\frac{N+1-2n}{N+1-n}
\end{align}

\item Binomial conjugation.
The diagonal matrix $B$ is defined by
$$B_{ii}=\binom{N}{i}\ .$$
Using the fact that $\FF{N}{}B$ is symmetric (ref.  \cite{FF}), i.e., $\FF{N}{}B=B(\FF{N}{})^*$, it follows
\begin{equation}\label{eq:binom}
\FF{N}{j\;i}=\binom{N}{i}^{-1}\binom{N}{j}\,\FF{N}{i\;j}\ .
\end{equation}
\item Sum of squares along a column, see \cite{FF} proof of Theorem 3.1.3
\begin{equation}\label{eq:colsquares}
\sum_{i=0}^N (\FF{N}{i\;j})^2=\binom{2N-2j}{N-j}\binom{2j}{j}\biggm/\binom{N}{j}
\end{equation}
We will see that\\

The sum of squares of column $j/2$ of $\FF{m}{}$ is $(-1)^{j/2}\FF{2m}{m\;j}\ .$\\

\item Sum of squares along a row, see \cite{FF} proof of Lemma 3.3.9
\begin{equation}\label{eq:rowsquares}
\sum_{j=0}^N (\FF{N}{i\;j})^2=\sum_{k=0}^i \binom{N+1}{2k+1}\binom{2k}{k}\binom{N-2k}{i-k}
\end{equation}
\end{enumerate}
\end{proposition}
\end{subsection}

We have a result for partial sums along a column without the squares:

\begin{theorem}\label{thm:sum}
For $j\ge2$, we have the partial sums
$$\sum_{n=0}^m (N-2n)\FF{N}{n\;j}=(N-j)\FF{N-1}{m\;j}+j\FF{N-1}{m\;j-2}=(N-1-2m)\FF{N-1}{m\;j-1}+\FF{N-1}{m\;j-2} \ .$$
\end{theorem}
\begin{proof}
Start with the Pascal relation Proposition \ref{prop:Pascal}, (ii), with $r=1$ and $N\rightarrow N-1$
$$\FF {N-1}{n  \;j} -\FF {N-1}{n-1  \;j} = \FF {N}{n\;j+1}$$
summing from $0$ to $m$ yields
\begin{equation}\label{eq:colsum}
\FF{N-1}{m\;j}=\sum_{n=0}^m \FF {N}{n\;j+1}
\end{equation}
and similarly with $j-2$ replacing $j$. Thus, summing the recurrence relation \eqref{eq:rrec} over $n$ we have
$$\sum_{n=0}^m (N-2n)\FF N{n\;j}= (N-j)\sum_{n=0}^m\FF N{n\;j+1}+j \sum_{n=0}^m\FF N{n\;j-1} $$
and using \eqref{eq:colsum} with $j$ adjusted accordingly yields the first equality. The second follows by applying the right hand side of the recurrence formula \eqref{eq:rrec}
with $N\rightarrow N-1$.
\end{proof}

Theorem \ref{thm:sqsum} takes the form
\begin{theorem}
For the symmetric Krawtchouk polynomials we have the sum of squares identity
$$\sum_{n=0}^m (N-2n)(\FF N{n\;j})^2=(N-j)(\FF {N-1}{m\;j})^2+j(\FF {N-1}{m\;j-1})^2\ .$$
\end{theorem}
\begin{subsection}{Special values. Catalan connection.}
\begin{proposition}
\begin{align}\label{eq:Cat1}
\Phi_{mj}^{2m} &=  \bbinom{m}{ j/2}\, (-1)^{j/2}\, \frac{\bbinom{2m}{m} }{\bbinom{2m}{j} } \ \text{for $j$ even\ }, \qquad 0 \text{\ for $j$ odd} \\
\mathstrut\nonumber \\
\Phi_{mj}^{2m+1} &= \bbinom{m}{ \lfloor j/2 \rfloor}\, (-1)^{\lfloor j/2\rfloor}\, \frac{\bbinom{2m+1}{m} }{\bbinom{2m+1}{j} }\ . \label{eq:Cat2}
\end{align}
\end{proposition}
\begin{proof}
Setting $N=2m$, $j=m$, the generating function, \eqref{eq:symmgf} becomes
\begin{align*}
(1+z)^m(1-z)^m&=(1-z^2)^m\\
&=\sum z^{2k}\binom{m}{k}(-1)^k=\sum z^\ell \FF{2m}{\ell\;m}
\end{align*}
hence the evaluation
$$
\FF{2m}{\ell\;m}=\binom{m}{\ell/2}(-1)^{\ell/2} \text{ for $\ell$ even }, \quad 0 \text{ for $\ell$ odd}
$$
and applying the binomial conjugation, eq.\,\eqref{eq:binom}, yields the result for $N=2m$.\\

For $N=2m+1$, $j=m$, we have
$$(1+z)(1-z^2)^m=\sum \left(z^{2k}\binom{m}{k}(-1)^k+z^{2k+1}\binom{m}{k}(-1)^k\right)$$
which yields
$$\Phi_{\ell m}^{2m+1} =\binom{m}{\lfloor \ell/2\rfloor}(-1)^{\lfloor \ell/2\rfloor}$$
and binomial conjugation completes the proof.
\end{proof}
\begin{remark}
Note that for $N$ even, these can be expressed as SuperCatalan numbers, in the terminology of \cite{GMT}, e.g., according to the relations
$$\binom{n}{k}\,\frac{\dbinom{2n}{n}}{\dbinom{2n}{2k}}=\frac{\dbinom{2n-2k}{n-k}\dbinom{2k}{k}}{\dbinom{n}{k}}
=\frac{(2n-2k)!\,(2k)!}{(n-k)!\,k!\,n!}$$ 
with $n$ replacing $m$ and $k$ replacing $j/2$ in \eqref{eq:Cat1}.
\end{remark}

Note that for $N=2m+1$, \eqref{eq:Cat2} yields for
the entry in the second column middle row
$$\Phi_{m 1}^{(2m+1)}=C_m=\frac{1}{m+1}\, \binom{2m}{m}\ . $$
the $m^{\rm th}$ Catalan number. In fact,
\begin{proposition} {\rm Catalan Connection}\\
We have the following evaluations involving Catalan	numbers.\\
1. For $N=2m$ even,
\begin{align*}
\Phi_{m-1,1}^{(2m)}&=C_m\\
\Phi_{m+1,1}^{(2m)}&=-C_m\\
\Phi_{m\;2}^{(2m)}&=-2C_{m-1}
\end{align*}

2. For $N=2m+1$ odd,
\begin{align*}
\Phi_{m\;1}^{(2m+1)}&=C_m\\
\Phi_{m\;2}^{(2m+1)}&=\Phi_{m+1\;1}^{(2m+1)}=\Phi_{m+1\;2}^{(2m+1)}=-C_m
\end{align*}

3. Reading  right-to-left along the rows yield evaluations correspondingly by sign symmetries.
\end{proposition}
\begin{proof}
The first two equations in \#1 follow from eq. \eqref{eq:columnone}. The third follows from \eqref{eq:Cat1}.\\

Similarly, for \#2, the column one evaluations follow from \eqref{eq:columnone} and the $\FF{2m+1}{m\;2}$ entry follows from \eqref{eq:Cat2}.
For $\FF{2m+1}{m+1\;2}$ ,  use Pascal as follows:
$$\Phi_{m+1\;1}^{(2m)}-\Phi_{m\;1}^{(2m)}=\Phi_{m+1\;2}^{(2m+1)}=-C_m$$
noting that $\Phi_{m\;1}^{(2m)}$ vanishes.
\end{proof}

\begin{remark} See \cite{FZ} for worksheets on Catalan numbers, listed up to $C_{20}$, and on Kravchuk matrices, listed up to $N=12$.
\end{remark}
\end{subsection}
\end{section}

\begin{section}{Dimensions of algebras over the Boolean lattice}

Let $\B$ denote the Boolean lattice of subsets of the standard $n$-set $\{1,2,\ldots,n\}$.
The layers, each consisting of subsets of cardinality $\ell$, are denoted $\B_\ell$. We identify
each element $i$ with a variable $e_i$, taken together forming the generators of a commutative algebra satisfying the conditions
$$e_i^2=0\ .$$
We call such variables \textit{zeons}.\\

A subset $\I\in\B_\ell$ is identified with the product
$$e_\I=e_{i_1}\cdots e_{i_\ell}$$
for $\I=\{i_1,\ldots,i_\ell\}$.\\

We will consider some algebras generated by the zeons and determine their
structure. See \cite{F1} for a full account.\\

Start with  the linear operator $\hat e_i$ of multiplication by $e_i$. 
$$ \hat e_i\, e_\I=\begin{cases} e_{\{i\}\cup \I}, & \text{if } i\notin\I\\
0, &\text{otherwise}
\end{cases}
$$

And for the dual basis $\{\delta_i\}$, the action of $\delta_i$ is given by the linear operator $\hat\delta_i$ defined by
$$ \hat \delta_i\, e_\I=\begin{cases} e_{\I\,\setminus\,\{i\}},  & \text{if } i\in\I\\
0, &\text{otherwise}
\end{cases}
$$

For convenience we will drop the $\hat{}$ notations.
With the standard inner product $\langle e_\I,e_\J \rangle=\delta_{\I \J}$, one checks that
$$\langle e_ie_\I,e_\J \rangle = \langle e_\I, e_i^*e_\J \rangle=\langle e_\I,\delta_ie_\J \rangle$$
the ${}^*$ indicating adjoint with respect to the inner product. We define the operator $T=\sum_i e_i$
and its adjoint 
$$T^*=\sum_i\delta_i=\sum e_i^*$$
Their commutator
$$U=[T^*,T]$$
has matrix elements
$$ U_{\I\J}=\begin{cases} n-2\ell, & \mathrm{ if }\  \I=\J \in\B_\ell  \\
0,& \mathrm{ otherwise }
\end{cases}
$$

Recall that a finite-dimensional ${}^*$-algebra is a direct sum of matrix algebras. 
Our goal here is to find the form of some algebras generated by these operators as direct sums of matrix algebras.
The description is provided by four numbers. The algebra consists of a direct sum of $m_i$ copies of matrix algebras of degree $d_i$. \bigskip

$d$ = Degree of the algebra = $\displaystyle \sum m_id_i$\bigskip

$\delta$ = Dimension of the algebra = $\displaystyle \sum d_i^2$ \bigskip

$\zeta$ = Dimension of the centralizer = $\displaystyle \sum m_i^2$ \bigskip

$z$ = Dimension of the center = the number of components of the decomposition \\

\begin{remark}
See, e.g., \cite[Ch.\,1]{Sag}, for an exposition in the context of group algebras.
\end{remark}

The calculations, although accessible by elementary means, are done illustrating the connection with Kravchuk matrices.\\

In the discussion below, we write $N=n+1$. \hfill\break 
For $N$ even, write $N=2m$, for $N$ odd, write $N=2m+1$. And
$$\left\lfloor \frac{n}{2} \right\rfloor =\begin{cases} m-1, &\text{ for $N$ even }\\ m, &\text{for $N$ odd}\end{cases}$$

\begin{subsection}{Algebra generated by $U$}
Since $U$ is diagonal, all of the $d_i=1$ and we only need to determine $m_i$. Since the $m_i$ are exactly given by the number
of sets in each layer, we see that the $m_i$ are precisely the binomial coefficients. See Appendix page for $U$.\\

Note that this is the "horizontal description" of $\mathcal{B}(n)$.\\

We have immediately
\begin{align*}
d&=2^n, \text{ degree of the algebra}\\
\delta&=n+1, \text{ dimension of the algebra}\\
          z&=n+1, \text{ dimension of the center}
\end{align*}
And we have the dimension of the centralizer
$$\zeta=\sum_i \binom{n}{i}^2=\binom{2n}{n}\ .$$
\end{subsection}

\begin{subsection}{Algebra generated by $T$ and $T^*$}
\hfill\\

This is the "vertical description" of $\mathcal{B}(n)$.\\

The $d_i$ may be seen directly to be $n+1-2\alpha$, where $0\le \alpha\le \lfloor n/2\rfloor$.
The $m_i$ are correspondingly given by $\binom{n}{\alpha}-\binom{n}{\alpha-1}$, with $m_0=1$.
See Appendix page for $n=4$ as well as that for  $T$ and $T^*$.\\

For the degree, we have
$$
\sum_{\alpha=0}^{\lfloor n/2\rfloor} \left[\binom{n}{\alpha}-\binom{n}{\alpha-1}\right]\,(n+1-2\alpha)
$$
Using the relation $(\Phi^N)^2=2^N\,I$, write this as
$$\sum_{\alpha=0}^m \FF{N}{1\;\alpha}\FF{N}{\alpha\;1}=\frac12\,(\FF{N}{1\;1})^2=\frac12\,2^N=2^n$$
appropriately.\\

And we have

\begin{align*}
d&=2^n, \text{ degree of the algebra}\\
\delta&=\sum_{\alpha=0}^{\lfloor n/2\rfloor} (n+1-2\alpha)^2=\binom{n+3}{3} \\
          z&=1+\lfloor n/2\rfloor, \text{ dimension of the center}
\end{align*}

\textsl{Proof of the formula for $\delta$:}\\

By sign symmetry, we have, using eq. \eqref{eq:rowsquares},
\begin{align*}
\sum_{\alpha=0}^m (n+1-2\alpha)^2&=\frac12\,\sum (\FF{N}{1\;\alpha})^2=\frac12\,\left[(N+1)N+2\,\binom{N+1}{3}\right] \\
&=\binom{N+2}{3}
\end{align*}
as required.\\

The dimension of the centralizer
$$\zeta=\sum_{\alpha=0}^{\lfloor n/2\rfloor} \left[\binom{n}{\alpha}-\binom{n}{\alpha-1}\right]^2=\frac{1}{n+1}\,\binom{2n}{n}=C_n\ .$$

\textsl{Proof of the formula for $\zeta$:}\\

Using eq. \eqref{eq:colsquares}, we have	
$$\sum_{\alpha=0}^m (\FF{N}{\alpha\;1})^2=\frac12\,\frac{\binom{2N-2}{N-1}\binom{2}{1}}{\binom{N}{1}}=C_{N-1}=C_n$$
as required.

\end{subsection}

\begin{subsection}{Algebra generated by $TT^*$ and $T^*T$}
\hfill\\

If the $\mathcal{A}(T,T^*)$ has $m_i$ and $d_i$, then here we have $d_i$ copies of $m_i$ all with $d=1$.\\

\begin{align*}
d&=2^n, \text{ degree of the algebra}\\
\delta&=\sum_{\alpha=0}^{\lfloor n/2\rfloor} (n+1-2\alpha)=
\begin{cases}
 (n+2)^2/4,&\text{if } $n$ \text{ is even}\\
\mathstrut \\
 (n+1)(n+3)/4&\text{if } $n$ \text{ is odd}\\
\end{cases}\\
          z&=1+\lfloor n/2\rfloor, \text{ dimension of the center}
\end{align*}

\textsl{Proof of the formula for $\delta$:}\\
\begin{align*} 
\sum_{\alpha=0}^{\lfloor\frac{n}{2}\rfloor} (n+1-2\alpha) &=\sum_0^m (N-2\alpha)\\
&=N(m+1)-m(m+1)=(m+1)(N-m)
\end{align*}
For $N$ odd we get $(m+1)^2=((N+1)/2)^2$. Even $N$ yields
$$(1+N/2)(N/2)=\frac{N(N+2)}{4}$$
as required.\\

And we have the dimension of the centralizer
$$\zeta=\sum_{\alpha=0}^{\lfloor n/2\rfloor} (n+1-2\alpha)\,\left[\binom{n}{\alpha}-\binom{n}{\alpha-1}\right]^2=
\begin{cases}
\displaystyle \binom{n}{n/2}^2,&\text{if } $n$ \text{ is even}\\
\mathstrut \\
\displaystyle 2\binom{n}{\lfloor n/2\rfloor}\binom{n-1}{\lfloor n/2\rfloor},&\text{if } $n$ \text{ is odd}\\
\end{cases}\ .$$
\end{subsection}

\textsl{Proof of the formula for $\zeta$:}\\

Using Theorem \ref{thm:sqsum}, we have
\begin{align*}
\sum_{\alpha=0}^m (N-2\alpha)\,(\FF{N}{\alpha\;1})^2&=(N-1)(\FF{N-1}{m\;1})^2+(\FF{N-1}{m\;0})^2\\
&=(N-1)\,\left( \binom{N-2}{m}-\binom{N-2}{m-1}\right)^2+\binom{N-1}{m}^2\\
&=(N-1)\binom{N-2}{m}^2\left(\frac{N-1-2m}{N-1-m}\right)^2+\binom{N-1}{m}^2\ .
\end{align*}
If $N$ is odd, $n=N-1=2m$ and we get $\dbinom{n}{n/2}^2$. If $N$ is even, substituting $N=2m$ and simplifying gives
$$2\,\binom{2m-1}{m-1}\,\binom{2m-2}{m-1}$$
and $n=N-1=2m-1$, $m-1=\lfloor n/2\rfloor$ yields the result.\\

\end{section}

\bibliographystyle{plain}

\newpage

\begin{appendix}
\begin{figure}
\centering
\hspace*{-1.3in}\includegraphics{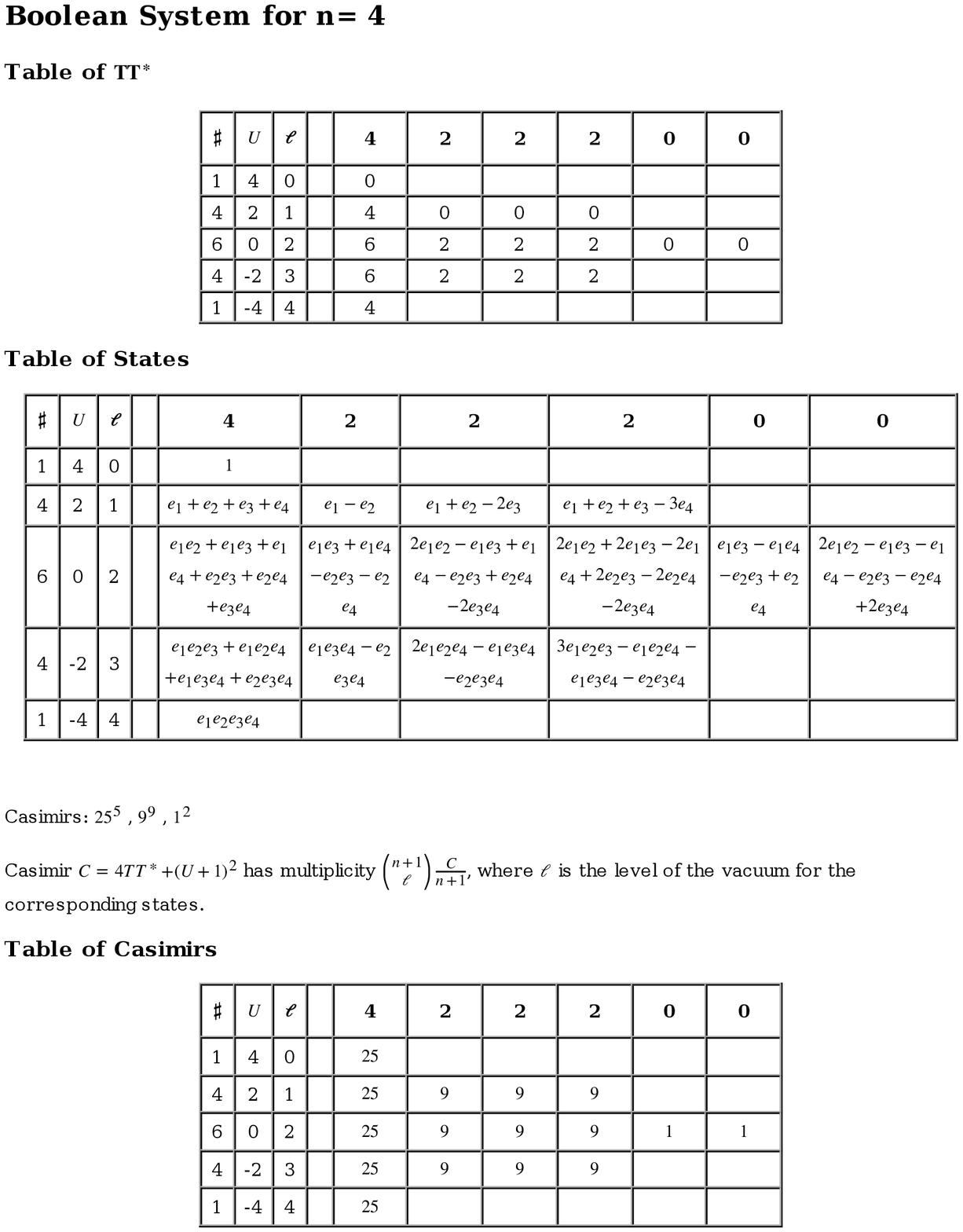}
\end{figure}

\newpage

\begin{figure}
\centering
\hspace*{-1.3in}\includegraphics{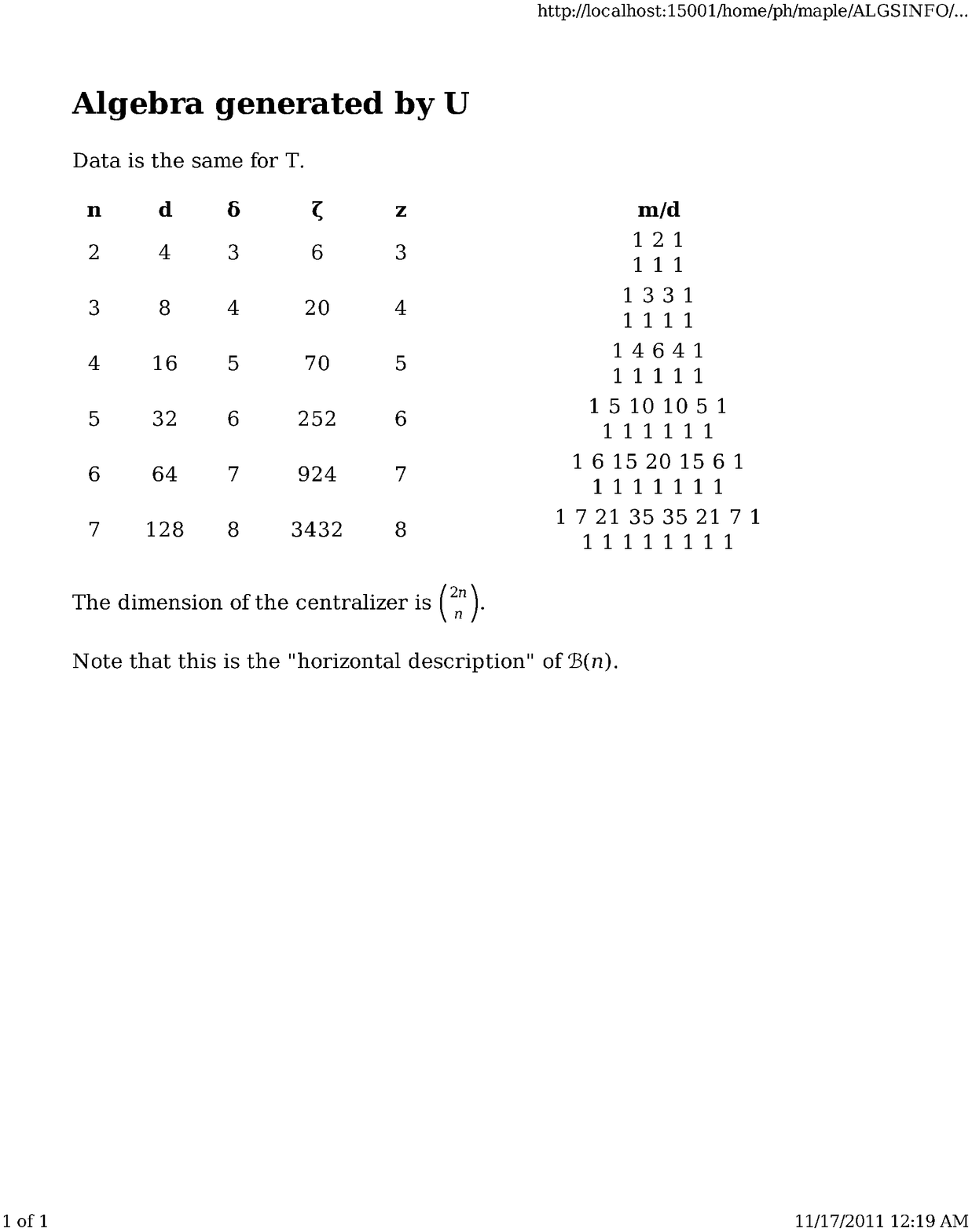}
\end{figure}

\newpage

\begin{figure}
\centering
\hspace*{-1.3in}\includegraphics{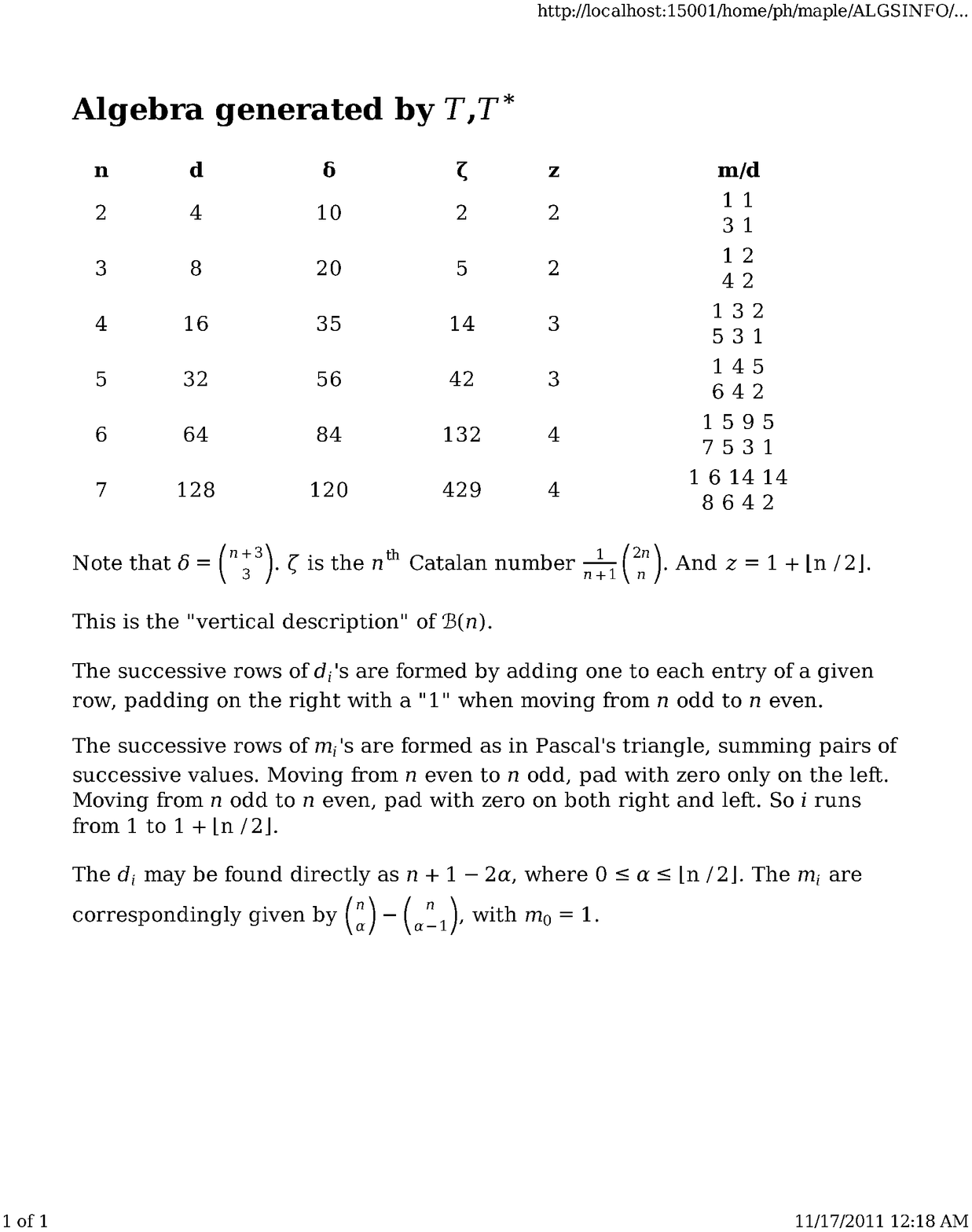}
\end{figure}

\newpage

\begin{figure}
\centering
\hspace*{-1.3in}\includegraphics{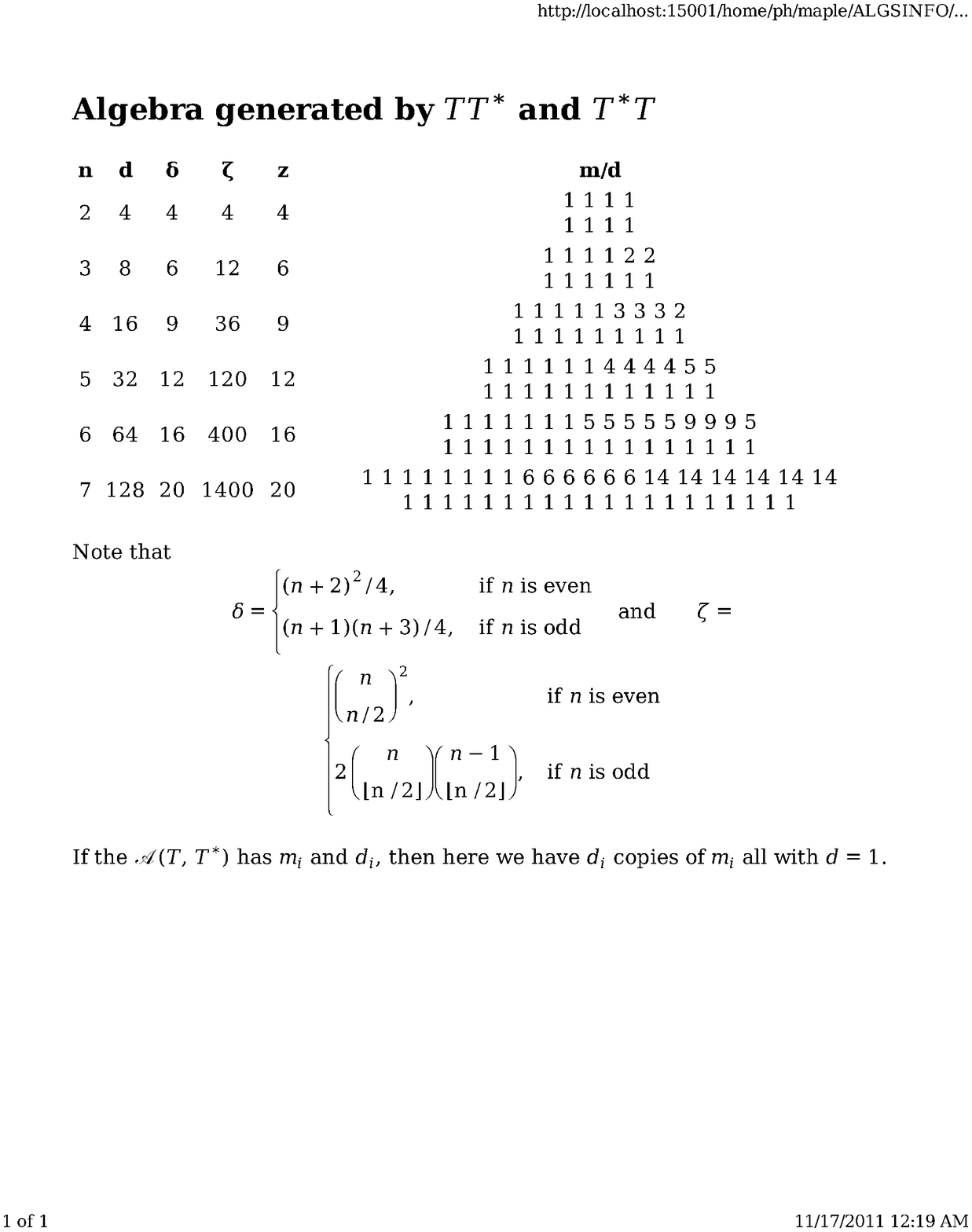}
\end{figure}

\newpage

\end{appendix}

\end{document}